\documentclass{article}
\usepackage{amssymb}
\usepackage{amsfonts}
\usepackage{graphicx}
\usepackage{amsmath}

\newtheorem{theorem}{Theorem}
\newtheorem{acknowledgement}[theorem]{Acknowledgement}

\newtheorem{lemma}[theorem]{Lemma}

\newtheorem{proposition}[theorem]{Proposition}
\newtheorem{remark}[theorem]{Remark}

\newenvironment{proof}[1][Proof]{\textbf{#1.} }{\ \rule{0.5em}{0.5em}}
\input{tcilatex}
\begin{document}

\title{Using the RD rational Arnoldi method for exponential integrators}
\author{Paolo Novati \\
Department of Pure and Applied Mathematics\\
University of Padova, Italy}
\maketitle

\begin{abstract}
In this paper we investigate some practical aspects concerning the use of
the Restricted-Denominator (RD) rational Arnoldi method for the computation
of the core functions of exponential integrators for parabolic problems. We
derive some useful a-posteriori bounds together with some hints for a
suitable implementation inside the integrators. Numerical experiments
arising from the discretization of sectorial operators are presented.
\end{abstract}

\section{Introduction}

For the solution of large stiff problems of the type%
\begin{equation}
u^{\prime }(t)=f(y(t))=Lu(t)+N(u(t)),  \label{pr}
\end{equation}%
where $L\in \mathbb{R}^{M\times M}$ arises from the discretization of
unbounded sectorial operators and $N$ is a nonlinear function, in recent
years much work has been done on the construction of exponential integrators
that might represent a promising alternative to classical solvers (see e.g. 
\cite{Minchev} or \cite{HO} for a comprehensive survey). As well known the
computation of the matrix exponential or related functions of matrices is at
the core of this kind of integrators. The main idea is to damp the stiffness
of the problem (assumed to be contained in $L$) on these computations so
that the integrator can be explicit.

Under the hypothesis that the functions of matrices involved are exactly
evaluated, the linear stability can be trivially achieved for both
Runge-Kutta and multistep based exponential integrators and hence highly
accurate and stable integrators can be constructed. On the other hand, the
main problem with this class of integrators is just the efficient
computation of such functions of matrices, so that, very few reliable codes
have been written (we remember the Rosenbrock type exponential integrators
presented in \cite{CO}, \cite{Holuse}, \cite{Nov}). For this reason many
authors are still doubtful about the potential of exponential integrators
with respect to classical implicit solvers even for semilinear problem of
type (\ref{pr}).

An exponential integrator requires at each time step the evaluation of a
certain number (depending on the accuracy) of functions of matrices of the
type $\varphi _{k}(hL)v$, where

\begin{eqnarray}
\varphi _{0}(h\lambda ) &=&\exp (h\lambda ),  \label{defi} \\
\varphi _{k+1}(h\lambda ) &=&\frac{\varphi _{k}(h\lambda )-\frac{1}{k!}}{%
h\lambda },\text{ for }k=0,1,2,...\text{ ,}  \notag
\end{eqnarray}%
being $h$ the time step. Actually this represents the general situation for
the Exponential Time Differencing methods, that is, the methods based on the
variation-of-constants formula; for Lawson's type method (also called
Integrating Factor methods) only the matrix exponential is involved. We
refer again to \cite{Minchev} and the reference therein for a background.

Among the existing techniques for the computation of functions of matrices
(we quote here the recent book of Higham \cite{Hig} for a survey), in this
context the Restricted-Denominator (RD) Rational Arnoldi algorithm
introduced independently in \cite{Vanh} and \cite{Morno} for the computation
of the matrix exponential seems to be an reliable approach. It is based on
the use of the so called RD rational forms, studied in \cite{Nors} for the
exponential function, 
\begin{equation*}
R_{i,j}(\lambda )=\frac{q_{i}(\lambda )}{(1-\delta \lambda )^{j}},\quad
\delta \in \mathbb{R}\mathbf{,}
\end{equation*}%
where $q_{i}$ is a polynomial of degree $\leq i.$ We refer again to \cite%
{Morno} for the basic references about the properties and the use of such
rational forms. While in the matrix case, the use of these approximants
requires the solution of linear systems with the matrix $(I-\delta L)$, as
shown in \cite{Nov} in the context of the solution of (\ref{pr}) when $L$ is
sectorial so typically sparse and well structured this linear algebra
drawback can be almost completely overtaken organizing suitably the
step-size control strategy and exploiting the properties of the RD Arnoldi
method concerning the choice of the parameter $\delta $. In other words the
number of linear systems to be solved can be drastically reduced with
respect to the total number of computations of functions of matrices
required by the integrator. Therefore the mesh independence property of the
method, that leads to a very fast convergence with respect to a standard
polynomial approach (see again \cite{Morno}), can be fully exploited for the
construction of competitive integrators.

A problem still open is that inside the integrator the rational Arnoldi
algorithm (responsible for most of the computational cost) have to be
supported by a robust and sharp error estimator. In the self-adjoint case
the problem has been treated in \cite{Mofi} where the author presents
effective a-posteriori error estimates, even in absence of information on
the location of the spectrum of $L$. Anyway, in the general case, when (\ref%
{pr}) arises for instance from the discretization of parabolic problems with
advection terms and/or non-zero boundary conditions the numerical range of $%
L $, that we denote by $F(L)$, may not reduce to a line segment. In this
sense the basic aim of this paper is to fill this gap providing error
estimates for the non-symmetric case using as few as possible information
about the location of $F(L)$. It is necessary to keep in mind that a
competitive code for (\ref{pr}) should also be able to update $L$
(interpreted as the Jacobian of $f$, \cite{Tok}, \cite{CO}) so that $F(L)$
is may be not fixed during the integration, and so it is important to reduce
as much as possible any pre-processing technique to estimate $F(L)$. In
particular assuming that $F(L)\subseteq \mathbb{C}^{-}$we shall provide
a-posteriori error estimates for the RD Arnoldi process using only
information about the angle of the sector containing $F(L)$, angle that is
typically independent of the sharpness of the discretization and hence
computable working in small dimension.

The paper is organized as follows. In Section 2 we present the basic idea of
the RD rational Arnoldi method and in Section 3 we derive some first general
error bounds based on the standard approaches. In Section 4, exploiting the
relation between the derivatives of the function $e^{1/z}$ and the Laguerre
polynomials extended to the complex plane, we derive some a-posteriori error
bounds. The problem of defining reliable a-priori bounds is investigated in
Section 5. Section 6 is devoted to the analysis of the generalized residual
as error estimator, that can be used to obtain information about the choice
of the parameter $\delta $ for the rational approximation. In Section 7 we
present some numerical examples arising from the discretization of a
one-dimensional advection-diffusion model. In Section 8 we provide some
hints about the use of the RD rational Arnoldi method inside an exponential
integrator with the aim of reducing as much as possible the number of
implicit computations of $(I-\delta L)^{-1}$. Finally, in Section 9 we
furnish a deeper analysis concerning the fast rate of convergence of the
method, that will provide further information about the optimal choice of
the parameter $\delta $.

\section{The RD rational Arnoldi method}

In what follows we denote by $\left\Vert \cdot \right\Vert $ the Euclidean
vector norm and its induced matrix norm. As already mentioned, the notation $%
F(L)$ indicates the \emph{numerical range} of $L$, that is, 
\begin{equation*}
F(L):=\left\{ \frac{x^{H}Lx}{x^{H}x},x\in \mathbb{C}^{M}\mathbf{\backslash }%
\left\{ 0\right\} \right\} ,
\end{equation*}%
while the spectrum of $L$ is denoted by $\sigma (L)$. The notation $\Pi _{m}$
indicates the space of the algebraic polynomials of degree $\leq m$.

Given $0\leq \theta <\frac{\pi }{2}$, let%
\begin{equation}
S_{\theta }=\left\{ \lambda :\left\vert \arg (-\lambda )\right\vert \leq
\theta \right\} \subset \mathbb{C}^{-}  \label{sect}
\end{equation}%
be the unbounded sector of the left half complex plane, symmetric with
respect to the real axis with vertex in $0$ and semiangle $\theta $. Let
moreover $\Gamma _{\theta }$ be the boundary of $S_{\theta }$. Throughout
the paper we assume that $F(L)\subset int(S_{\theta })$, the interior of $%
S_{\theta }$. Accordingly, $L$ is a so-called sectorial operator (see e.g. 
\cite{Ka} Chap. V, for a background).

Given a vector $v\in \mathbb{R}^{M}$, with $\left\Vert v\right\Vert =1$,
consider the problem of computing 
\begin{equation}
y^{(k)}=\varphi _{k}(hL)v,  \label{cpsi}
\end{equation}%
where $\varphi _{k}$ is defined by (\ref{defi}). The RD rational approach
seeks for approximations to $\varphi _{k}(h\lambda )$ of the type 
\begin{equation*}
R_{m-1,m-1}(\lambda )=\frac{p_{k,m-1}(\lambda )}{(1-\delta \lambda )^{m-1}},%
\mathbf{\quad }p_{k,m-1}(\lambda )\in \Pi _{m-1},\quad m\geq 1,
\end{equation*}%
where $\delta >0$ is a suitable parameter. Turning to the matrix case, $%
y^{(k)}$ is approximated by elements of the Krylov subspaces

\begin{equation*}
K_{m}(Z,v)=span\left\{ v,Zv,Z^{2}v,...,Z^{m-1}v\right\} ,\quad m\geq 1,
\end{equation*}%
with respect to $v$ and the matrix $Z$ defined by the transform 
\begin{equation*}
Z=(I-\delta L)^{-1}.
\end{equation*}%
In this sense the idea is to use a polynomial method to compute $%
y^{(k)}=f_{k}(Z)v$, where%
\begin{equation*}
f_{k}(z):=\varphi _{k}\left( \frac{h}{\delta }\left( 1-\frac{1}{z}\right)
\right)
\end{equation*}%
is singular at $0$.

For the construction of the subspaces $K_{m}(Z,v)$ we employ the classical
Arnoldi method. As is well known it generates an orthonormal sequence$\
\left\{ v_{j}\right\} _{j\geq 1}$, with $v_{1}=v$, such that $%
K_{m}(Z,v)=span\left\{ v_{1},v_{2},...,v_{m}\right\} $. Moreover, for every $%
m$, 
\begin{equation}
ZV_{m}=V_{m}H_{m}+h_{m+1,m}v_{m+1}e_{m}^{H},  \label{i1}
\end{equation}%
where $V_{m}=\left[ v_{1},v_{2},...,v_{m}\right] $, $H_{m}$ is upper
Hessenberg matrix with entries $h_{i,j}=v_{i}^{H}Zv_{j}$ and $e_{j}$ is the $%
j$-th vector of the canonical basis of \ $\mathbb{R}^{m}$.

The $m$-th RD-rational Arnoldi approximation to $y^{(k)}$ is defined as (see 
\cite{Knizh})%
\begin{equation}
y_{m}^{(k)}=V_{m}f_{k}(H_{m})e_{1}.  \label{kn}
\end{equation}%
It can be seen that 
\begin{equation}
y_{m}^{(k)}=\overline{p}_{k,m-1}(Z)v,  \label{int}
\end{equation}%
where $\overline{p}_{k,m-1}\in $ $\Pi _{m-1}$ interpolates, in the Hermite
sense, the function $f_{k}(z)$ in the eigenvalues of $H_{m}$ (see \cite%
{Saad2}).

As mentioned in the Introduction this technique has been introduced
independently in \cite{Vanh} and \cite{Morno}. Anyway, the idea of using
rational Krylov approximations to matrix functions was originally introduced
in \cite{Druki}. More recently this approach has been extended to the case
of multiple poles and is commonly referred to as RKS (Rational Krylov
Subspace) approximation (see \cite{knisi}, \cite{posi}, \cite{BR}).

\section{General error bounds}

Before stating a general error bound for the method, we need to locate $F(Z)$%
. Consider the function $\chi (\lambda )=(1-\delta \lambda )^{-1}$. Denoting
by $D_{1/2,1/2}$ the disk centered in $1/2$ with radius $1/2$, let%
\begin{equation}
G_{\theta }=\{z:z=\chi (\lambda ),\lambda \in S_{\theta }\}\subseteq
D_{1/2,1/2}.  \label{gt}
\end{equation}%
Its boundary, $\Sigma _{\theta }$, is made by two circular arcs meeting with
angle $2\theta $ at $0$ and $1$. Regarding the field of values of $Z$, $F(Z)$%
, we can state the following result that will be used frequently throughout
the paper.

\begin{proposition}
\label{pint}If $F(L)\subset int(S_{\theta })$ then $F(Z)\subset
int(G_{\theta }).$
\end{proposition}

\begin{proof}
Obviously $\sigma (Z)=\chi (\sigma (L))$, so $F(Z)$ cannot lie entirely
outside $G_{\theta }$. Now assume that there exists $\lambda \in \Gamma
_{\theta }$ such that $\chi (\lambda )\in F(Z)$, that is, $F(Z)\cap \Sigma
_{\theta }\neq \varnothing $. Hence, there exists $y\in \mathbb{C}^{M}$, $%
\left\Vert y\right\Vert =1$, such that%
\begin{equation}
y^{H}\left( I-\delta L\right) ^{-1}y=\frac{1}{1-\delta \lambda }.
\label{rel1}
\end{equation}%
Defining $x:=\left( I-\delta L\right) ^{-1}y$ we easily obtain%
\begin{equation*}
x^{H}\left( I-\delta L^{T}\right) x=\frac{1}{1-\delta \lambda },
\end{equation*}%
and hence%
\begin{equation}
1-\delta \frac{x^{H}L^{T}x}{x^{H}x}=\frac{1}{\left( 1-\delta \lambda \right)
\left\Vert x\right\Vert ^{2}}.  \notag
\end{equation}%
By (\ref{rel1}) we have%
\begin{equation}
\left\Vert x\right\Vert \left\vert 1-\delta \lambda \right\vert \geq 1.
\label{frel}
\end{equation}%
Now let us define $\mu :=\frac{x^{H}L^{T}x}{x^{H}x}\in F(L)$. We have%
\begin{equation}
\left\Vert x\right\Vert ^{2}=\left( 1-\delta \lambda \right) ^{-1}\left(
1-\delta \mu \right) ^{-1},  \label{xn}
\end{equation}%
and hence%
\begin{equation*}
\func{Im}\left( \left( 1-\delta \lambda \right) ^{-1}\left( 1-\delta \mu
\right) ^{-1}\right) =0,
\end{equation*}%
that implies%
\begin{equation}
\frac{\func{Im}\left( \left( 1-\delta \lambda \right) ^{-1}\right) }{\func{Re%
}\left( \left( 1-\delta \lambda \right) ^{-1}\right) }=-\frac{\func{Im}%
\left( \left( 1-\delta \mu \right) ^{-1}\right) }{\func{Re}\left( \left(
1-\delta \mu \right) ^{-1}\right) }.  \label{cmp}
\end{equation}%
Now since $\left( 1-\delta \mu \right) ^{-1}\in int(G_{\theta })$ and $%
\left( 1-\delta \lambda \right) ^{-1}\in \Sigma _{\theta }$, by (\ref{cmp})
it must be $\func{Re}\left( \left( 1-\delta \lambda \right) ^{-1}\right) >%
\func{Re}\left( \left( 1-\delta \mu \right) ^{-1}\right) $ and $\left\vert 
\func{Im}\left( \left( 1-\delta \lambda \right) ^{-1}\right) \right\vert
>\left\vert \func{Im}\left( \left( 1-\delta \mu \right) ^{-1}\right)
\right\vert $ so that 
\begin{equation*}
\left\vert 1-\delta \mu \right\vert ^{-1}<\left\vert 1-\delta \lambda
\right\vert ^{-1}.
\end{equation*}%
Using this relation, by (\ref{xn}) we finally have%
\begin{equation*}
\left\Vert x\right\Vert ^{2}\left\vert 1-\delta \lambda \right\vert ^{2}<1,
\end{equation*}%
that contradicts (\ref{frel}). Since the field of values is connected the
proof is complete.
\end{proof}

\begin{remark}
\label{rem1}In order to provide information about the geometry of $F(Z)$, it
is worth referring to\ \cite{Zac} Theorem 5.2 in which the author proves
that if $L$ is an invertible matrix then%
\begin{equation*}
\lim_{s\rightarrow \infty }\left( \frac{1}{F((L-sI)^{-1})}+s\right) =F(L).
\end{equation*}%
Taking $\delta =1/s$, we have that for small values of $\delta $%
\begin{equation*}
F((I-\delta L)^{-1})\approx \frac{1}{1-\delta F(L)}.
\end{equation*}
\end{remark}

Going back to our method, the corresponding error $%
E_{k,m}:=y^{(k)}-y_{m}^{(k)}$ can be expressed and bounded in many ways (we
quote here the recent papers \cite{BR} and \cite{MRD} for a background on
the error estimates for both polynomial and rational Arnoldi approximation
to matrix functions). The following proposition states a general result.

\begin{proposition}
\label{stp}Let $G\subseteq D_{1/2,1/2}$ be a compact such that $F(Z)\subset
int(G)$ and whose boundary $\Sigma $ is a rectifiable Jordan curve. For
every $p_{m-1}\in \Pi _{m-1}$%
\begin{equation}
\left\Vert E_{k,m}\right\Vert \leq \frac{1}{2\pi }\int\nolimits_{\Sigma }%
\frac{\left\vert f_{k}(z)-p_{m-1}(z)\right\vert }{dist(z,F(Z))}\left\Vert
v-\left( zI-Z\right) V_{m}\left( zI-H_{m}\right) ^{-1}e_{1}\right\Vert
\left\vert dz\right\vert .  \label{ge}
\end{equation}
\end{proposition}

\begin{proof}
Using the properties of the Arnoldi algorithm we know that for every $%
p_{m-1}\in \Pi _{m-1}$, 
\begin{equation*}
V_{m}p_{m-1}(H_{m})e_{1}=p_{m-1}(Z)v.
\end{equation*}%
Hence from this identity it follows that, for $m\geq 1$ 
\begin{equation}
E_{k,m}=f_{k}(Z)v-p_{m-1}(Z)v-V_{m}(f_{k}(H_{m})-p_{m-1}(H_{m}))e_{1}.
\label{fe}
\end{equation}%
Now since $F(H_{m})\subseteq F(Z)$ we can write (\ref{fe}) in the
Dunford-Taylor integral form%
\begin{equation*}
E_{k,m}=\frac{1}{2\pi i}\int\nolimits_{\Sigma }\left(
f_{k}(z)-p_{m-1}(z)\right) \left[ \left( zI-Z\right) ^{-1}v-V_{m}\left(
zI-H_{m}\right) ^{-1}e_{1}\right] dz.
\end{equation*}%
Collecting $\left( zI-Z\right) ^{-1}$ and using (see \cite{Spik})%
\begin{equation*}
\left\Vert \left( zI-Z\right) ^{-1}\right\Vert \leq \frac{1}{dist(z,F(Z))},
\end{equation*}%
we prove (\ref{ge}).
\end{proof}

Now since%
\begin{equation*}
v-\left( zI-Z\right) V_{m}\left( zI-H_{m}\right) ^{-1}e_{1}=\frac{q_{m}(Z)v}{%
q_{m}(z)},
\end{equation*}%
where%
\begin{equation*}
q_{m}(z)=\det (zI-H_{m}),
\end{equation*}%
(see \cite{MN1}), any bound for $\left\Vert q_{m}(Z)v\right\Vert /\left\vert
q_{m}(z)\right\vert $ and any choice for $G$ and $p_{m-1}$ leads to a bound
for $\left\Vert E_{k,m}\right\Vert $. This technique has been used for
instance in \cite{Morno} and \cite{LubHoch}. In particular in \cite{Morno}
the authors use the relation%
\begin{equation}
\left\Vert q_{m}(Z)v\right\Vert =\prod\limits_{j=1}^{m}h_{j+1,j},
\label{try}
\end{equation}%
and the inequality%
\begin{equation}
\left\vert q_{m}(z)\right\vert \geq dist(z,F(Z))^{m}.  \label{lb}
\end{equation}

Going back to our situation, the main problem is that if we simply assume
that $F(L)\subset S_{\theta }$ (in other words $F(L)$ arbitrarily large) we
have that $dist(z,F(Z))\rightarrow 0$ as $z\rightarrow 0$ ($\func{Re}\lambda
\rightarrow -\infty $) because we have to consider the singularity of $f_{k}$
at $0$. Therefore using a lower bound like (\ref{lb}) (but the situation
remains true even for other approaches (cf. \cite{LubHoch})) terms of the
type $f_{k}(z)/z^{m+1}$ would appear in (\ref{ge}). In the exponential case (%
$k=0$) this is not a problem because $f_{0}(z)/z^{m+1}\rightarrow 0$ for $%
z\rightarrow 0$, but for $k>0$ the situation changes completely since 
\begin{equation*}
\frac{f_{k}(z)}{z^{m+1}}\approx \frac{\delta }{h(k-1)!}\frac{1}{z^{m}}
\end{equation*}%
for $z\rightarrow 0$.

Because of the difficulties just explained, our approach for deriving error
bounds is not based on the use of the Cauchy integral formula. Exploiting
the interpolatory nature of the standard Arnoldi method, we notice, as
pointed out also in \cite{Eier}, that the error can be expressed in the form 
\begin{equation}
E_{k,m}=g_{k,m}(Z)q_{m}(Z)v,  \label{erapr}
\end{equation}%
where (cf. (\ref{int}))%
\begin{equation}
g_{k,m}(z):=\frac{f_{k}(z)-\overline{p}_{k,m-1}(z)}{\det (zI-H_{m})}.
\label{gkm}
\end{equation}%
In \cite{Eier} this relationship is used as the basis for the construction
of restarted methods for the computation of matrix functions.

We can state the following basic result that will be used throughout the
paper and that allows to overcome the difficulties of working with formula (%
\ref{ge}).

\begin{proposition}
\label{pg}Let $F(L)\subset S_{\theta }$ and let $\tau :=h/\delta $. Then%
\begin{equation}
\left\Vert E_{k,m}\right\Vert \leq K\frac{1}{\tau ^{k}(m+k)!}\max_{z\in
G_{\theta }}\left\vert \frac{d^{m+k}}{dz^{m+k}}f_{0}(z)z^{k}\right\vert
\prod\nolimits_{i=1}^{m}h_{i+1,i},  \label{erb}
\end{equation}%
where $2\leq K\leq 11.08$. In the symmetric case we can take $K=1$.
\end{proposition}

\begin{proof}
By \cite{Cro} we know that%
\begin{equation*}
\left\Vert g_{k,m}(Z)\right\Vert \leq K\max_{z\in F(Z)}\left\vert
g_{k,m}(z)\right\vert ,
\end{equation*}%
and hence by (\ref{try}) and (\ref{erapr}) 
\begin{equation*}
\left\Vert E_{k,m}\right\Vert \leq K\max_{z\in F(Z)}\left\vert
g_{k,m}(z)\right\vert \prod\nolimits_{i=1}^{m}h_{i+1,i}.
\end{equation*}%
Now, by induction one proves that for $k\geq 1$ 
\begin{equation}
f_{k}(z)=\frac{f_{0}(z)z^{k}-s_{k-1}(z)z}{\tau ^{k}(z-1)^{k}},  \label{fk}
\end{equation}%
where $s_{0}(z)=1$ and%
\begin{equation*}
s_{k}(z)=s_{k-1}(z)z+\frac{\tau ^{k}(z-1)^{k}}{k!}\in \Pi _{k}\text{ for }%
k\geq 1.
\end{equation*}%
Putting (\ref{fk}) in (\ref{gkm}) we obtain%
\begin{equation*}
g_{k,m}(z)=\frac{f_{0}(z)z^{k}-s_{k-1}(z)z-\tau ^{k}(z-1)^{k}\overline{p}%
_{k,m-1}(z)}{\tau ^{k}(z-1)^{k}\det (zI-H_{m})}.
\end{equation*}%
Now, the polynomial $\tau ^{k}(z-1)^{k}\overline{p}_{k,m-1}(z)\in
\prod\nolimits_{m+k-1}$ interpolates in the Hermite sense the function $%
f_{0}(z)z^{k}-s_{k-1}(z)z$ in the eigenvalues of $H_{m}$ and in $z=1$.
Henceforth $g_{k,m}(z)$ is a divided difference that can be bounded using
the Hermite-Genocchi formula (see e.g. \cite{deb}), so that 
\begin{equation*}
\left\vert g_{k,m}(z)\right\vert \leq \frac{1}{\tau ^{k}(m+k)!}\max_{\xi \in
co(\left\{ z,\sigma (H_{m}),1\right\} )}\left\vert \frac{d^{m+k}}{d\xi ^{m+k}%
}f_{0}(\xi )\xi ^{k}\right\vert ,
\end{equation*}%
where $co(\left\{ z,\sigma (H_{m}),1\right\} $ denotes the convex hull of
the point set given by $z$, $\sigma (H_{m})$ and $1$. Since $\sigma
(H_{m})\subset F(Z)$, and $F(Z)\subset G_{\theta }$ by Proposition \ref{pint}%
, the result follows.
\end{proof}

\section{A posteriori error estimates}

By (\ref{erb}), in order to provide a-posteriori error estimates we just
need to study the derivatives of the function $f_{0}(z)z^{k}$. We need to
introduce the generalized Laguerre polynomials, defined by 
\begin{equation*}
L_{n}^{(\alpha )}(z)=\sum\limits_{j=0}^{n}(-1)^{j}\binom{n+\alpha }{n-j}%
\frac{z^{j}}{j!}.
\end{equation*}%
We can state the following result.

\begin{lemma}
Let $\tau =\frac{h}{\delta }$. For $m\geq 1$ 
\begin{equation}
\frac{1}{\tau ^{k}(m+k)!}\frac{d^{m+k}}{dz^{m+k}}f_{0}(z)z^{k}=\frac{%
(-1)^{m+1}\tau }{z^{m+k+1}}f_{0}(z)\frac{(m-1)!}{(m+k)!}L_{m-1}^{(k+1)}(%
\frac{\tau }{z}).  \label{ede}
\end{equation}
\end{lemma}

\begin{proof}
First of all remember that $f_{0}(z)=e^{\tau }e^{-\tau /z}$. Defining $%
\omega =z/\tau $ and using Rodrigues' formula for Laguerre polynomials (see 
\cite{Aste} p.101) we obtain%
\begin{eqnarray*}
\frac{d^{m+k}}{dz^{m+k}}\exp (-\frac{\tau }{z})z^{k} &=&\frac{1}{\tau ^{m}}%
\frac{d^{m+k}}{d\omega ^{m+k}}\exp \left( -\omega ^{-1}\right) \left( \omega
^{-1}\right) ^{-k}, \\
&=&\frac{1}{\tau ^{m}}(-1)^{m+k}(m+k)!\exp (-\omega ^{-1})\omega
^{-m}L_{m+k}^{(-1-k)}(\omega ^{-1}).
\end{eqnarray*}%
The result arises from the relation (see \cite{Mag} p.240)%
\begin{equation*}
L_{m+k}^{(-1-k)}(\frac{\tau }{z})=(-1)^{k+1}(\frac{\tau }{z})^{k+1}\frac{%
(m-1)!}{(m+k)!}L_{m-1}^{(k+1)}(\frac{\tau }{z}).
\end{equation*}
\end{proof}

Before stating the main result we need to remember the following properties
of the generalized Laguerre polynomials, that can be found in \cite{Aste}
pp. 785-786.

\begin{enumerate}
\item[L1] 
\begin{equation*}
L_{n}^{(\alpha +\beta +1)}(z_{1}+z_{2})=\sum_{j=0}^{n}L_{j}^{(\alpha
)}(z_{1})L_{n-j}^{(\beta )}(z_{2}).
\end{equation*}

\item[L2] 
\begin{equation*}
L_{n}^{(\alpha )}(z_{1}z_{2})=\sum_{j=0}^{n}\binom{n+\alpha }{j}%
L_{j}^{(\alpha )}(z_{1})z_{2}^{j}(1-z_{2})^{n-j}.
\end{equation*}

\item[L3] 
\begin{equation*}
\exp (\frac{-x}{2})\left\vert L_{n}^{(\alpha )}(x)\right\vert \leq \frac{%
\Gamma (n+\alpha +1)}{n!\Gamma (\alpha +1)},\quad \text{for }x\geq 0.
\end{equation*}
\end{enumerate}

\begin{proposition}
Given $r\geq 0$, let $z=\left( 1+\delta re^{i\theta }\right) ^{-1}\in \Sigma
_{\theta }$. Let moreover%
\begin{equation}
c_{j}(\theta ):=\left( 1+\sqrt{2(1-\cos \theta )}\right) ^{j}.  \label{cj}
\end{equation}%
Then%
\begin{eqnarray}
\left\vert L_{m-1}^{(k+1)}(\frac{\tau }{z})\right\vert &\leq &e^{\frac{hr}{2}%
}\sum_{j=0}^{m-1}\left\vert L_{m-1-j}^{(k)}(\tau )\right\vert c_{j}(\theta ),
\label{Lk1} \\
&\leq &e^{\frac{\tau +hr}{2}}\sum\nolimits_{j=0}^{m-1}\binom{m+k-j-1}{k}%
c_{j}(\theta ).  \label{Lk2}
\end{eqnarray}
\end{proposition}

\begin{proof}
For $z=\left( 1+\delta re^{i\theta }\right) ^{-1}$%
\begin{equation*}
\frac{\tau }{z}=\tau +hre^{i\theta },\quad r\geq 0.
\end{equation*}%
Using L1 with $\alpha =k$, $\beta =0$, $z_{1}=\tau $ and $z_{2}=hre^{i\theta
}$, and then L2 with $z_{1}=hr$ and $z_{2}=e^{i\theta }$, we have%
\begin{eqnarray}
\left\vert L_{m-1}^{(k+1)}(\frac{\tau }{z})\right\vert &=&\left\vert
\sum_{j=0}^{m-1}L_{m-j-1}^{(k)}(\tau )L_{j}^{(0)}(hre^{i\theta })\right\vert
,  \label{pri} \\
&\leq &\sum_{j=0}^{m-1}\left\vert L_{m-j-1}^{(k)}(\tau )\right\vert
\sum_{s=0}^{j}\left\vert L_{s}^{(0)}(hr)\right\vert \left\vert \binom{j}{s}%
e^{is\theta }(1-e^{i\theta })^{j-s}\right\vert .  \notag
\end{eqnarray}%
Since 
\begin{equation*}
\sum_{s=0}^{j}\left\vert \binom{j}{s}e^{is\theta }(1-e^{i\theta
})^{j-s}\right\vert =c_{j}(\theta ),
\end{equation*}%
formulas (\ref{Lk1}) and (\ref{Lk2}) are obtained applying L3 to $%
L_{s}^{(0)}(hr)$ and then to $L_{m-j-1}^{(k)}(\tau )$.
\end{proof}

\begin{theorem}
\label{pro1}Assume that $F(L)\subset S_{\theta }$, with $\theta <\frac{\pi }{%
3}$. Then%
\begin{eqnarray}
\left\Vert E_{k,m}\right\Vert &\leq &K\frac{e^{\tau \left( \cos \theta -%
\frac{1}{2}\right) -m-k-1}}{\tau ^{m+k}}\left( \frac{2(m+k+1)}{2\cos \theta
-1}\right) ^{m+k+1}C_{k,m}(\tau ,\theta )\prod\limits_{i=1}^{m}h_{i+1,i},
\label{fe1} \\
&\leq &K\frac{e^{\tau \cos \theta -m-k-1}}{\tau ^{m+k}}\left( \frac{2(m+k+1)%
}{2\cos \theta -1}\right) ^{m+k+1}C_{k,m}^{\prime }(\theta
)\prod\limits_{i=1}^{m}h_{i+1,i},  \label{fe2}
\end{eqnarray}%
where%
\begin{eqnarray}
C_{k,m}(\tau ,\theta ) &:&=\frac{(m-1)!}{(m+k)!}\sum_{j=0}^{m-1}\left\vert
L_{m-1-j}^{(k)}(\tau )\right\vert c_{j}(\theta ),  \label{ckm} \\
C_{k,m}^{\prime }(\theta ) &:&=\frac{(m-1)!}{(m+k)!}\sum\nolimits_{j=0}^{m-1}%
\binom{m+k-j-1}{k}c_{j}(\theta ),  \label{dkm}
\end{eqnarray}%
and $K$ defined as in Proposition \ref{pg}.
\end{theorem}

\begin{proof}
For $z\in \Sigma _{\theta }$%
\begin{equation*}
\frac{1}{z}=1+\delta re^{i\theta },\quad r\geq 0,
\end{equation*}%
and%
\begin{equation*}
f_{0}(z)=e^{\tau -\frac{\tau }{z}}=e^{-hre^{i\theta }}.
\end{equation*}%
Hence, using (\ref{erb}), (\ref{ede}) and (\ref{Lk1}) we obtain%
\begin{eqnarray}
\left\Vert E_{k,m}\right\Vert &\leq &K\max_{r\geq 0}\left\vert e^{-hr(\cos
\theta -\frac{1}{2})}\left( 1+\delta re^{i\theta }\right)
^{m+k+1}\right\vert \tau  \notag \\
&&\times \frac{(m-1)!}{(m+k)!}\sum_{j=0}^{m-1}\left\vert
L_{m-1-j}^{(k)}(\tau )\right\vert c_{j}(\theta
)\prod\limits_{i=1}^{m}h_{i+1,i}.  \label{tmo}
\end{eqnarray}%
Since for $\theta <\pi /3$%
\begin{equation*}
e^{-hr\left( \cos \theta -\frac{1}{2}\right) }\left( 1+\delta r\right)
^{m+k+1}\leq \frac{e^{\tau \left( \cos \theta -\frac{1}{2}\right) -m-k-1}}{%
\tau ^{m+k+1}}\left( \frac{2(m+k+1)}{2\cos \theta -1}\right) ^{m+k+1},
\end{equation*}%
(looking for the maximum with respect to $r$), we immediately obtain (\ref%
{fe1}). Using again (\ref{erb}) and (\ref{ede}) but now with (\ref{Lk2}) we
arrive at the coarser bound (\ref{fe2}).
\end{proof}

\begin{remark}
While formulas (\ref{fe1}) and (\ref{fe2}) theoretically hold for $\theta <%
\frac{\pi }{3}$ since $h_{m+1,m}=0$ for $m\leq M$, it is necessary to point
out that for $\theta \approx \frac{\pi }{3}\ $we may observe a rapid growth
of the term%
\begin{equation*}
\left( \frac{1}{2\cos \theta -1}\right)
^{m+k+1}\prod\limits_{i=1}^{m}h_{i+1,i},
\end{equation*}%
depending of course on the problem, so that the bounds may be useless. This
situation is caused by the bound (\ref{Lk1}) that leads to the appearance of
the term $2\cos \theta -1$ at the denominator. Working in inexact
arithmetics the situation is even more difficult because of the loss of
orthogonality of the vectors $v_{j}$ of the Arnoldi algorithm and hence the
accumulation of errors on the entries $h_{i+1,i}$. For these reasons, in
practice, formulas (\ref{fe1}) and (\ref{fe2}) should be used only for $%
\theta $ not much close to $\frac{\pi }{3}$.
\end{remark}

\begin{remark}
For the exponential case ($k=0$) we have%
\begin{equation*}
C_{0,m}^{\prime }(\theta )=\frac{1}{m}\sum\nolimits_{j=0}^{m-1}c_{j}(\theta
),
\end{equation*}%
and hence by (\ref{fe2})%
\begin{equation}
\left\Vert E_{0,m}\right\Vert \leq K\frac{e^{\tau \cos \theta -m-1}}{m\tau
^{m}}\left( \frac{2(m+1)}{2\cos \theta -1}\right)
^{m+1}\sum\nolimits_{j=0}^{m-1}c_{j}(\theta )\prod\limits_{i=1}^{m}h_{i+1,i}.
\label{xp}
\end{equation}
\end{remark}

\begin{remark}
In the self-adjoint case $\theta =0$ we have $c_{j}(\theta )=1$ and formula (%
\ref{fe2}) simplifies to%
\begin{equation*}
\left\Vert E_{k,m}\right\Vert \leq K\frac{e^{\tau -m-k-1}}{\tau ^{m+k}}\frac{%
\left( 2(m+k+1)\right) ^{m+k+1}}{(k+1)!}\prod\limits_{i=1}^{m}h_{i+1,i}.
\end{equation*}
\end{remark}

The reason for which we consider two bounds in Theorem \ref{pro1} is that
the second one (\ref{fe2}) allows us to define suitably the parameter $\tau $
(and then $\delta $) while the first one (\ref{fe1}) should be used whenever 
$\tau $ has been defined. Indeed, assuming $\prod%
\nolimits_{i=1}^{m}h_{i+1,i} $ independent of $\delta $ and then of $\tau $
(actually this is not true as we explain in Section 9) by (\ref{fe2}),
looking for the minimum of $e^{\tau \cos \theta }\tau ^{-\left( m+k\right) }$
we easily find that the optimal value for $\tau $ is given by%
\begin{equation}
\tau =\frac{m+k}{\cos \theta }.  \label{topt}
\end{equation}

The a-posteriori bounds provided by Theorem \ref{pro1} depends substantially
on the semiangle $\theta $ of the sector containing $F(L)$. Therefore, the
most natural way to proceed is to compute the boundary of $F(L)$ using the
standard codes available in literature (as for instance the Matlab code 
\texttt{fv.m} by Higham \cite{Hi}). It is important to observe that $\theta $
is generally independent of the discretization so that one can work in
smaller dimension.

\bigskip

While the hypothesis $F(L)\subset S_{\theta }$ of Theorem \ref{pro1} is
extremely general, the underlying assumption is that $L$ represents an
arbitrary sharp discretization of an unbounded operator. On the other side,
if it is known that $F(L)$ is contained in a bounded sector then Proposition %
\ref{stp} can be used to derive sharper error estimates. In general we may
refer again to \cite{BR} and the references therein for a background on the
most used techniques based on the use of the integral representation of the
error.

Anyway, here we want also to show how to adapt our approach in precence of
more information on $F(L)$. Let $D_{0,R}$ be the disk centered at $0$ with
radius $R$, and assume that $F(L)\subset S_{\theta }\cap D_{0,R}$. Using
again (\ref{erb}) and (\ref{ede}), we arrive at the bound%
\begin{eqnarray}
\left\Vert E_{k,m}\right\Vert  &\leq &K\max_{0\leq s\leq hR}\left\vert
e^{-s\cos \theta }\left( 1+\frac{s}{\tau }\right)
^{m+k+1}L_{m-1}^{(k+1)}(\tau +se^{i\theta })\right\vert \tau   \notag \\
&&\times \frac{(m-1)!}{(m+k)!}\prod\limits_{i=1}^{m}h_{i+1,i}.  \label{bdn}
\end{eqnarray}%
In order to define a suitable value for $\tau $, we just need to bound the
Laguerre polynomials as in (\ref{Lk2}), so that the optimal value is
obtained ooking for the minimum of%
\begin{equation*}
\tau \left( 1+\frac{s}{\tau }\right) ^{m+k+1}e^{\frac{\tau }{2}}.
\end{equation*}%
A good approximation for this minimum is given by%
\begin{equation}
\tau =\sqrt{2hR(m+k+1)},  \label{tbo}
\end{equation}%
that is obtained considering the bound%
\begin{equation*}
\left( 1+\frac{s}{\tau }\right) ^{m+k+1}\leq \exp \left( (m+k+1)\frac{hR}{%
\tau }\right) .
\end{equation*}%
Using this value of $\tau $ we can derive practical error bounds seeking for
the maximum of the function $\left\vert e^{-s\cos \theta }\left( 1+\frac{s}{%
\tau }\right) ^{m+k+1}L_{j}^{(0)}(se^{i\theta })\right\vert $ (cf. (\ref{pri}%
)) in the interval $[0,hR]$.

\section{A-priori error bounds}

Formula (\ref{topt}) obviously requires to know the number of iterations
that are necessary to achieve a certain accuracy. In this sense we need to
bound in some way $\prod\nolimits_{i=1}^{m}h_{i+1,i}$. By (\ref{try}) and
since%
\begin{equation*}
\left\Vert q_{m}(Z)v\right\Vert \leq \left\Vert p_{m}(Z)v\right\Vert
\end{equation*}%
for each monic polynomial $p_{m}$ of exact degree $m$ (see \cite{Trefe} p.
269), a bound for $\prod\nolimits_{i=1}^{m}h_{i+1,i}$ can be stated using
Faber polynomials as explained in \cite{Beck}, that leads to%
\begin{equation}
\prod\nolimits_{i=1}^{m}h_{i+1,i}=\left\Vert q_{m}(Z)v\right\Vert \leq
2\gamma (G)^{m},  \label{be}
\end{equation}%
where $\gamma (G)$ is the logarithmic capacity of a compact$\ G\ $containing 
$F(Z)$ and where $f_{k}$ is analytic.

\begin{proposition}
\label{pap}Let $\theta ^{\ast }=0.48124$ and assume that $F(L)\subset
S_{\theta }$, with $\theta <\theta ^{\ast }$. Then for $\tau =(m+k)/\cos
\theta $ 
\begin{equation}
\left\Vert E_{k,m}\right\Vert \leq 11K\rho (\theta )^{m},  \label{ap}
\end{equation}%
where%
\begin{equation}
\rho (\theta ):=\left( 1+\sqrt{2(1-\cos \theta )}\right) \frac{\cos \theta }{%
4\cos \theta -2}\frac{\pi }{\pi -\theta }<1\quad \text{for\quad }0\leq
\theta <\theta ^{\ast }.  \label{rot}
\end{equation}
\end{proposition}

\begin{proof}
Since $F(Z)\subset G_{\theta }$ by Proposition \ref{pint}, let us consider
the compact subset $G=G_{\theta }$. The associated conformal mapping 
\begin{equation*}
\psi :\mathbb{C}\backslash \left\{ w:\left\vert w\right\vert \leq 1\right\}
\rightarrow \mathbb{C}\backslash G_{\theta },
\end{equation*}%
is given by%
\begin{eqnarray}
\psi (w) &=&\frac{(w+1)^{2-\nu }}{(w+1)^{2-\nu }-(w-1)^{2-\nu }},  \notag \\
&=&\frac{1}{2(2-\nu )}w+\frac{1}{2}+\frac{1}{6}\frac{\left( 1-\nu \right)
\left( 3-\nu \right) }{2-\nu }\frac{1}{w}+O\left( \frac{1}{w^{2}}\right) ,
\label{le}
\end{eqnarray}%
where $\nu =2\theta /\pi $. The coefficient of the leading term of the
Laurent expansion (\ref{le}) is the logarithmic capacity, so that by (\ref%
{be}) we have%
\begin{equation}
\prod\nolimits_{i=1}^{m}h_{i+1,i}\leq 2\left( \frac{1}{2(2-\nu )}\right)
^{m}.  \label{bh}
\end{equation}%
Inserting this bound in (\ref{fe2}) we easily obtain for $\theta <\frac{\pi 
}{3}$%
\begin{eqnarray*}
\left\Vert E_{k,m}\right\Vert &\leq &K\frac{e^{\tau \cos \theta -m-k-1}}{%
\tau ^{m+k}}\left( \frac{m+k+1}{2\cos \theta -1}\right)
^{m+k+1}2^{-m+k+2}\left( \frac{\pi }{\pi -\theta }\right)
^{m}C_{k,m}^{\prime }(\theta ), \\
&\leq &K\frac{m+k+1}{\cos \theta }\left( \frac{\cos \theta }{2\cos \theta -1}%
\right) ^{m+k+1}2^{-m+k+2}\left( \frac{\pi }{\pi -\theta }\right)
^{m}C_{k,m}^{\prime }(\theta ),
\end{eqnarray*}%
where the second inequality arises from the choice $\tau =(m+k)/\cos \theta $%
.

Now, by the definition (\ref{cj}), it is rather easy to show that%
\begin{eqnarray*}
C_{k,m}^{\prime }(\theta ) &=&\frac{(m-1)!}{(m+k)!}\sum\nolimits_{j=0}^{m-1}%
\binom{m+k-j-1}{k}c_{j}(\theta ), \\
&\leq &\frac{1}{k!(m+k)}\left( \frac{m-1}{m+k-1}\right) ^{m}c_{m}(\theta )
\end{eqnarray*}%
so that%
\begin{equation}
\left\Vert E_{k,m}\right\Vert \leq K\frac{e^{-k}}{k!\cos \theta }\left( 
\frac{\cos \theta }{2\cos \theta -1}\right) ^{k+1}2^{k+3}\left[ \rho (\theta
)\right] ^{m}.  \label{app}
\end{equation}%
Since $1/2\leq \Phi (\theta )<1$ for $0\leq \theta <\theta ^{\ast }$, and
since for each $k\geq 0$%
\begin{equation*}
\frac{e^{-k}}{k!\cos \theta }\left( \frac{\cos \theta }{2\cos \theta -1}%
\right) ^{k+1}2^{k+3}\leq \frac{8}{\cos \theta ^{\ast }}\left( \frac{\cos
\theta ^{\ast }}{2\cos \theta ^{\ast }-1}\right) =10.351
\end{equation*}%
the proof is complete.
\end{proof}

\begin{remark}
Proposition \ref{pap} shows the mesh-independence of the method for $\theta
<\theta ^{\ast }$ since the bound (\ref{ap}) is independent of the
discretization of the underlying sectorial operator. In Section 9 this
considaration is extended to $\theta <\frac{\pi }{3}$. By (\ref{app}) and (%
\ref{rot}), in the self-adjoint case ($K=1$) the bound (\ref{ap}) reads%
\begin{equation*}
\left\Vert E_{k,m}\right\Vert \leq \frac{8}{k!}\left( \frac{2}{e}\right)
^{k}\left( \frac{1}{2}\right) ^{m}.
\end{equation*}
\end{remark}

It is worth noting that by (\ref{fe}) for every $p_{m-1}\in \Pi _{m-1}$ we
have that 
\begin{equation*}
\left\Vert E_{k,m}\right\Vert \leq 2K\max_{z\in G}\left\vert
f_{k}(z)-p_{m-1}(z)\right\vert ,
\end{equation*}%
where we assume that $G\subset D_{1/2,1/2}$ is compact, connected, with
associated conformal mapping $\phi $, and such that $F(Z)\subset G$.
Therefore, in principle, one could try to derive a-priori error bounds
choosing suitably the polynomial sequence $\left\{ p_{m-1}\right\} _{m\geq
1} $. Anyway, the classical results in complex polynomial approximation
state that even taking $\left\{ p_{m-1}\right\} _{m\geq 1}$ as a sequence of
polynomials that asymptotically behaves as the sequence of polynomial of
best uniform approximation of $f_{k}$ on $G$ (see e.g \cite{Smile} for a
theoretical background and examples) we have%
\begin{equation*}
\left[ \max_{z\in G}\left\vert f_{k}(z)-p_{m-1}(z)\right\vert \right]
^{1/m}\rightarrow \frac{1}{R}\text{\quad as }m\rightarrow \infty ,
\end{equation*}%
where $R$ is such that $\phi (-R)=0$, since $f_{k}$ is singular at 0 (\emph{%
maximal convergence} property, see e.g \cite{Wal} Chapter IV). The main
problem is that assuming $L$ to be unbounded, $0\in G$ and consequently $R=1$%
.

For this reasons, in our opinion the only reasonable approach to derive
a-priori error bounds, is to define $\left\{ p_{m-1}\right\} _{m\geq 1}$ as
a sequence of polynomials interpolating $f_{k}$ at point belonging to $G$,
and then to use the Hermite-Genocchi formula to bound the divided
differences. Using this formula and taking for instance $p_{m-1}$ as the
sequence of interpolants at the zeros of Faber polynomials we just obtain
the error bound given in Proposition \ref{pap} (see \cite{MN1}).

\section{The generalized residual}

By the integral representation of function of matrices and (\ref{kn}), we
know that the error can be written as%
\begin{equation}
E_{k,m}=\frac{1}{2\pi i}\int_{\Sigma _{\theta
}}f_{k}(z)[(zI-Z)^{-1}v-V_{m}(zI-H_{m})^{-1}e_{1}]dz.  \label{fot1}
\end{equation}%
In order to monitor the approximations during the computation we can
consider the so-called generalized residual \cite{Holuse}, defined as 
\begin{equation}
R_{k,m}=\frac{1}{2\pi i}\int_{\Gamma }f_{k}(z)r_{m}(z)dz,  \label{res}
\end{equation}%
which is obtained from (\ref{fot1}) by replacing the error%
\begin{equation*}
(zI-Z)^{-1}v-V_{m}(zI-H_{m})^{-1}e_{1}
\end{equation*}%
with the corresponding residual 
\begin{equation*}
r_{m}(z)=v-(zI-Z)V_{m}(zI-H_{m})^{-1}e_{1}.
\end{equation*}%
Using the fundamental relation (\ref{i1}) we have immediately%
\begin{equation*}
r_{m}(z)=h_{m+1,m}(e_{m}^{H}(zI-H_{m})^{-1}e_{1})v_{m+1},
\end{equation*}%
and inserting this relation in (\ref{res}) we obtain 
\begin{equation*}
R_{k,m}=h_{m+1,m}(e_{m}^{H}f_{k}(H_{m})e_{1})v_{m+1},
\end{equation*}%
so that we may assume 
\begin{equation}
E_{k,m}\approx \left\Vert R_{k,m}\right\Vert =h_{m+1,m}\left\vert
e_{m}^{H}f_{k}(H_{m})e_{1}\right\vert .  \label{gr}
\end{equation}

In order to show the reliability of this approximation let us consider the
operator

\begin{equation}
Lu=-u^{\prime \prime }+cu^{\prime },\quad c\geq 0,  \label{lu}
\end{equation}%
discretized with central differences in $[0,1]$ with uniform mesh $h=1/(M+1)$%
, and Dirichelet boundary conditions. For our examples, we consider the
computation of $\varphi _{k}(hL)v$ for $k=1,2$, with $v=(1,...,1)^{T}/\sqrt{M%
}$, comparing the exact error and the generalized residual. We take $M=1000$%
, $h=0.1$, and we consider the cases of $c=2$ and $c=4$, whose corresponding
sector semiangles are respectively $\theta =0.201$ and $\theta =0.425$. We
define $\tau =15/\cos \theta $. The results, collected in Figure 1, shows
the accuracy of the approximation (\ref{gr}).

It is necessary to point out that the use of (\ref{gr}) has the basic
disadvantage that it requires the computation of $f_{k}(H_{m})$, $m=1,2,...$%
, and this is a computational drawback whenever a great amount of matrix
functions evaluations are required to integrate a certain problem, even if $%
m $ can be considered much smaller than $M$. Moreover, it frequently happens
(as in our experiments) that the generalized residual tends to underestimate
the error during the first iterations, and this can be particularly
dangerous when computing $\varphi _{k+1}(hL)v$\ with $\left\Vert
v\right\Vert \ll 1$, as for instance in the case of the computation of the
internal stages of an exponential Runge-Kutta method, in which $\left\Vert
v\right\Vert =O(h)$.

On the other side, exploiting the mesh independence of the method the
generalized residual can be successfully used to estimate the optimal value
for the parameter $\tau $, that is $\tau _{opt}=\left( m+k\right) /\cos
\theta $. In other words, using a coarser discretization of the operator we
look for the value of $m$ such that using the corresponding $\tau _{opt}$ we
obtain a certain tolerance in exactly $m$ iterations. For the experiments
reported in Figure 1 we considered the discretization of (\ref{lu}) with
only $M=50$ internal points, observing in both cases that $\left\vert
R_{m}\right\vert \leq 1e-12$ for $m\geq 13$. For this reason we have chosen $%
\tau =15/\cos \theta $.

\begin{center}
\FRAME{itbpFU}{5.2823in}{1.9873in}{0in}{\Qcb{Figure 1 - Comparison between
the exact error and the generalized residual for problem (\protect\ref{lu})
with $h=0.1$. In both experiments $\protect\tau =15/\cos \protect\theta $.}}{%
}{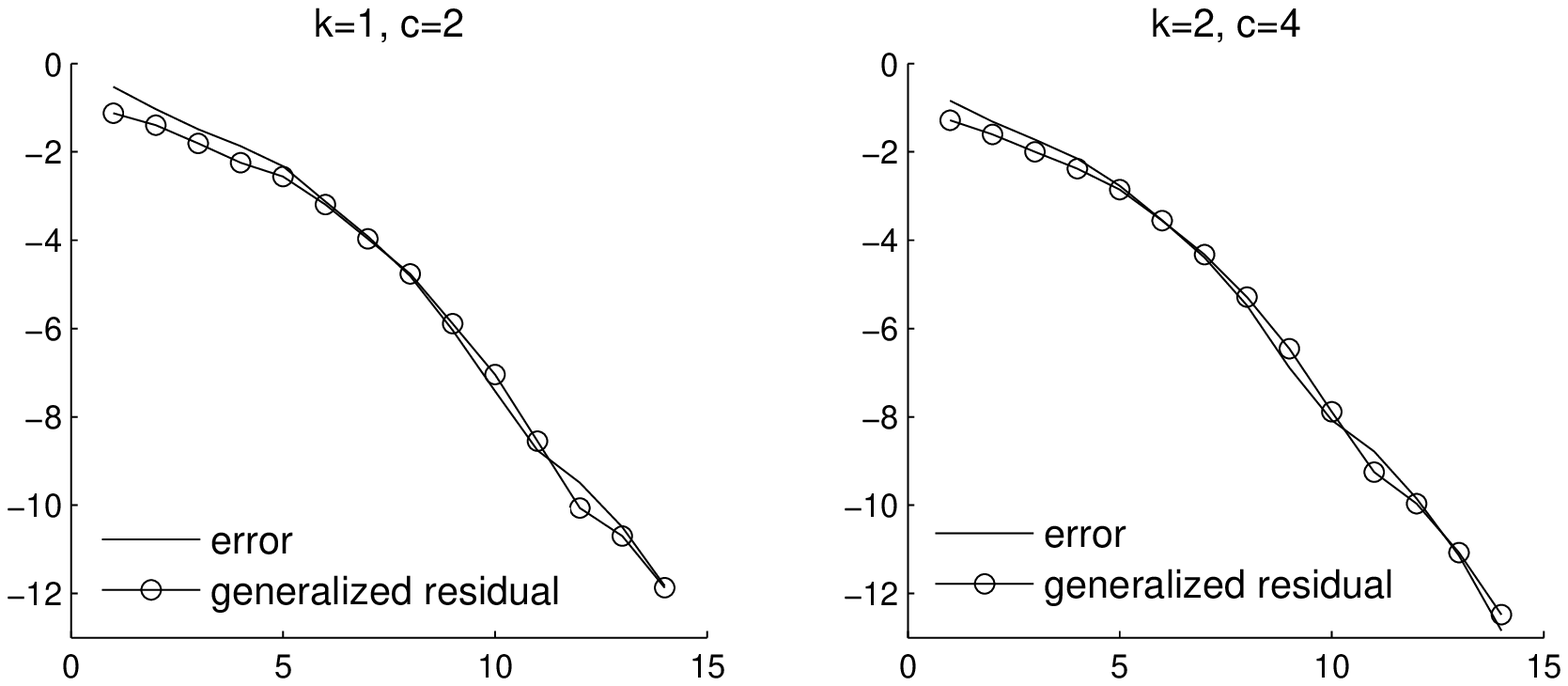}{\special{language "Scientific Word";type
"GRAPHIC";maintain-aspect-ratio TRUE;display "USEDEF";valid_file "F";width
5.2823in;height 1.9873in;depth 0in;original-width 8.0851in;original-height
3.1834in;cropleft "0";croptop "1";cropright "1";cropbottom "0";filename
'genres.eps';file-properties "XNPEU";}}
\end{center}

\section{Numerical experiments for the a-posteriori error bound}

For our numerical experiments we consider again the operator (\ref{lu}),
discretized as in previous section. We consider the computation of the
functions $\varphi _{k}(hL)v$, with $v$ as before and $k=0,1,2$, for $h=0.5$
(Figure 2) and $h=0.05$ (Figure 3). In all examples we do not consider the
symmetric case corresponding to $c=0$ (already investigated in \cite{Mofi}),
but only the cases $c=2$ and $c=4$. As before, for the choice of $\tau $ we
examined the behavior of the method for the coarser discretization of the
same operator with only $M=50$ interior points, thus exploiting the mesh
independence of the method. The analysis suggested to take $\tau =8/\cos
\theta $ for all experiments with $h=0.5$ and $\tau =15/\cos \theta $ for
those with $h=0.05$, thus independently of the function and $c$, using the
tolerance $1e-12$. Inside the Arnoldi iterations the vectors $Zv_{j}$, $%
j\geq 1$ (cf. Section 2), are computed via the LU factorization of $I-\delta
L$. The error bound is given by (\ref{fe1}).

\begin{center}
\FRAME{itbpFU}{4.8784in}{4.9856in}{0in}{\Qcb{Figure 2 - Error and error
bound (\protect\ref{fe1}) for $k=0,1,2$, $h=0.5$, $L$ arising from (\protect
\ref{lu}) with $c=2$ and $c=4$.}}{}{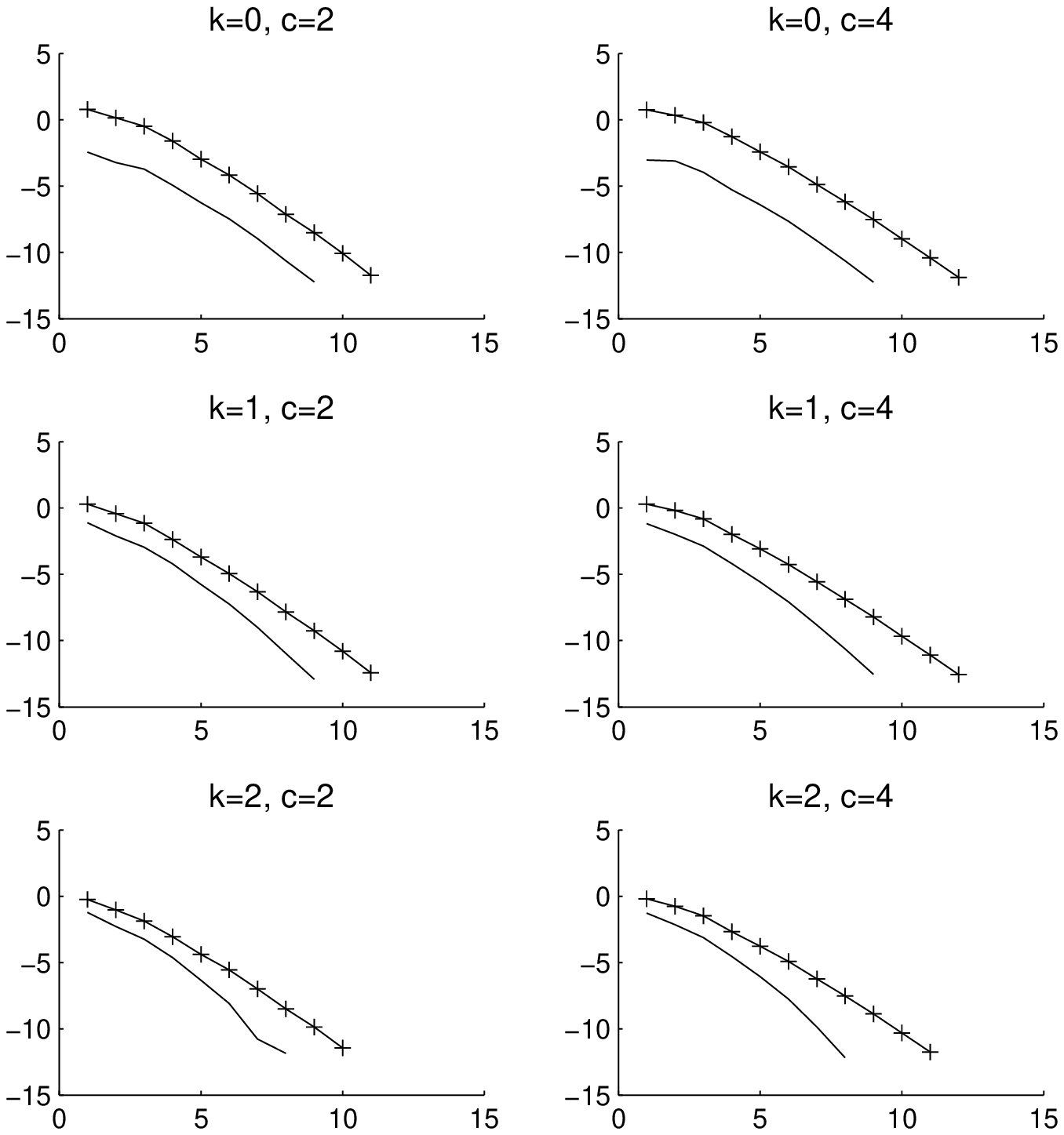}{\special{language "Scientific
Word";type "GRAPHIC";maintain-aspect-ratio TRUE;display "USEDEF";valid_file
"F";width 4.8784in;height 4.9856in;depth 0in;original-width
6.4688in;original-height 6.6098in;cropleft "0";croptop "1";cropright
"1";cropbottom "0";filename 't005.eps';file-properties "XNPEU";}}

\FRAME{itbpFU}{4.5662in}{5.0799in}{0in}{\Qcb{Figure 3 - Error and error
bound (\protect\ref{fe1}) for $k=0,1,2$, $h=0.05$, $L$ arising from (\protect
\ref{lu}) with $c=2$ and $c=4$. }}{}{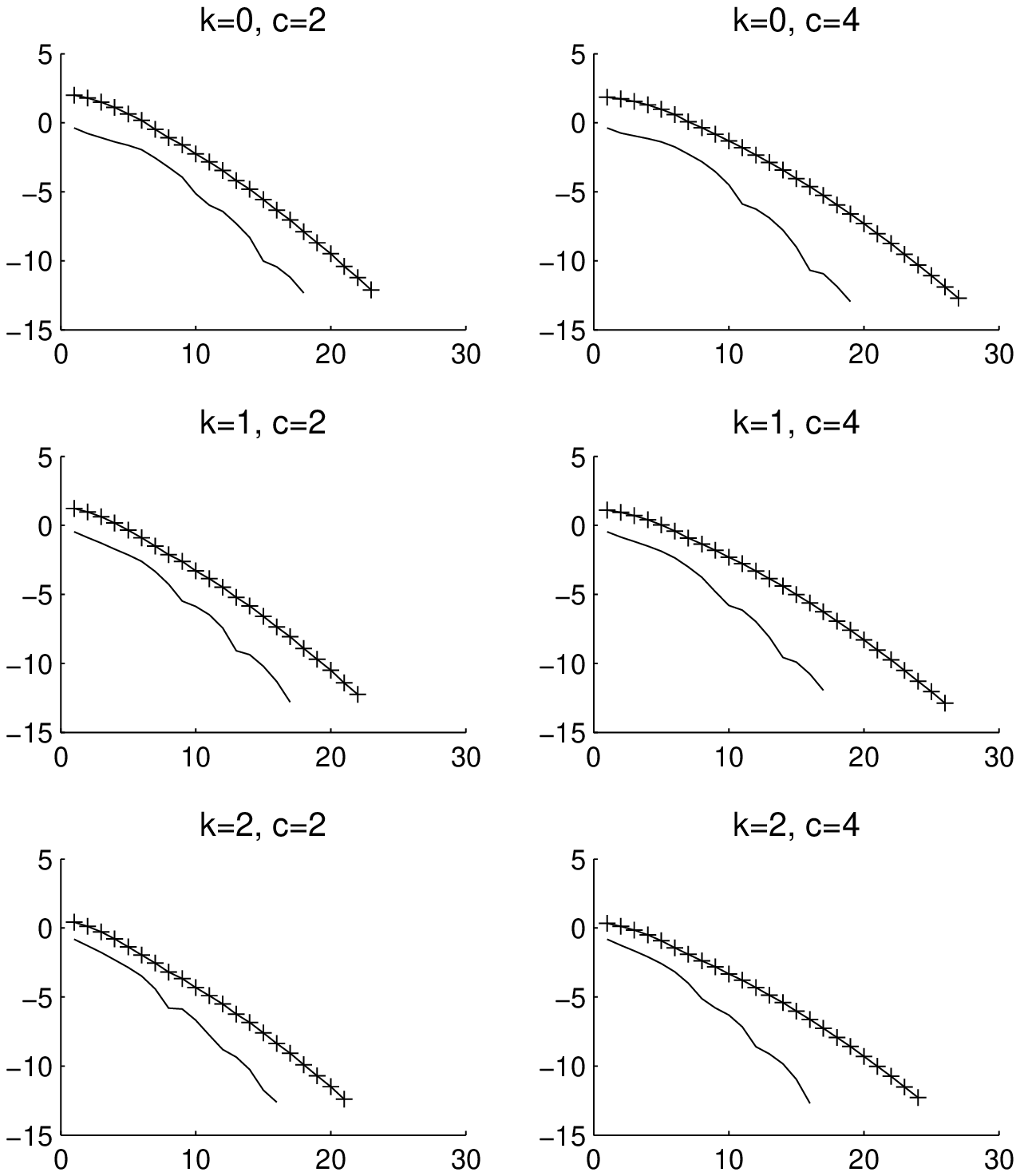}{\special{language "Scientific
Word";type "GRAPHIC";maintain-aspect-ratio TRUE;display "USEDEF";valid_file
"F";width 4.5662in;height 5.0799in;depth 0in;original-width
6.0502in;original-height 6.7369in;cropleft "0";croptop "1";cropright
"1";cropbottom "0";filename 't05.eps';file-properties "XNPEU";}}
\end{center}

Comparing Figure 2 with Figure 3 we can observe that the method tends to
become slower reducing $h$. The reason is that for small values of $h$, the
rate of the decay of the singular values of $Z$ becomes slower and this
reduces the rate of the decay of $\prod\nolimits_{i=1}^{m}h_{i+1,i}$. A
deeper analysis of this behavior will be presented in Section 9.

\section{Non-optimal choice of $\protect\tau $}

Employing the RD Arnoldi method inside an exponential integrator requires
some considerations. First of all, in our opinion the method can be used
only if the implicit computation of $Z$ can be obtained with a sparse
factorization technique. The use of an inner-outer iteration can be too much
expensive in this context. Indeed, the basic point is that organizing
suitably the code one can heavily reduce the number of factorizations of $%
I-\delta L$ (see e.g \cite{Nov}), because the method seems to be really
robust with respect to the choice of $\tau $. For this reason we want here
to show what happens taking $\tau $ even quite far from the optimal one.

For simplicity (the situation is representative of what happens in general)
let us assume to work with exponential function and $\theta =0$. We assume
moreover that the corresponding bound (\ref{xp}) (in which $c_{j}(\theta )=1$%
, $j\geq 0$) is equal to a prescribed tolerance for a certain $m$ with the
theoretical optimal choice $\tau _{opt}=m$. We seek for the interval $%
I_{m,n}=[\tau _{1},\tau _{2}]$ such that for $\tau \in I_{m,n}$ the number
of iterations necessary to achieve the same tolerance is at most $n$ ($\geq
m $). Using (\ref{xp}) and the approximation $h_{m+1,m}\approx 1/4$ ($m>1$)
that is obtained forcing the equal sign in the a-priori bound (\ref{bh}), in
Figure 4 we can observe the result for $n=m+1,m+2$. For each $m$ the
corresponding extremal points $\tau _{1}$ and $\tau _{2}$ of the intervals $%
I_{m,m+1}$ and $I_{m,m+2}$ are plotted. These points are obtained solving
with respect to $\tau $ the equation (cf. (\ref{xp}))%
\begin{equation*}
\frac{e^{\tau -n-1}}{\tau ^{n}}(2(n+1))^{n+1}\prod%
\nolimits_{i=1}^{n}h_{i+1,i}=\frac{e^{-1}}{m^{m}}(2(m+1))^{m+1}\prod%
\nolimits_{i=1}^{m}h_{i+1,i},
\end{equation*}%
for $n=m+1,m+2$.

\begin{center}
\FRAME{itbpFU}{4.7314in}{3.4541in}{0in}{\Qcb{Figure 4 - Boundary of the
region $I_{m,m+1}$ and $I_{m,m+2}$.}}{}{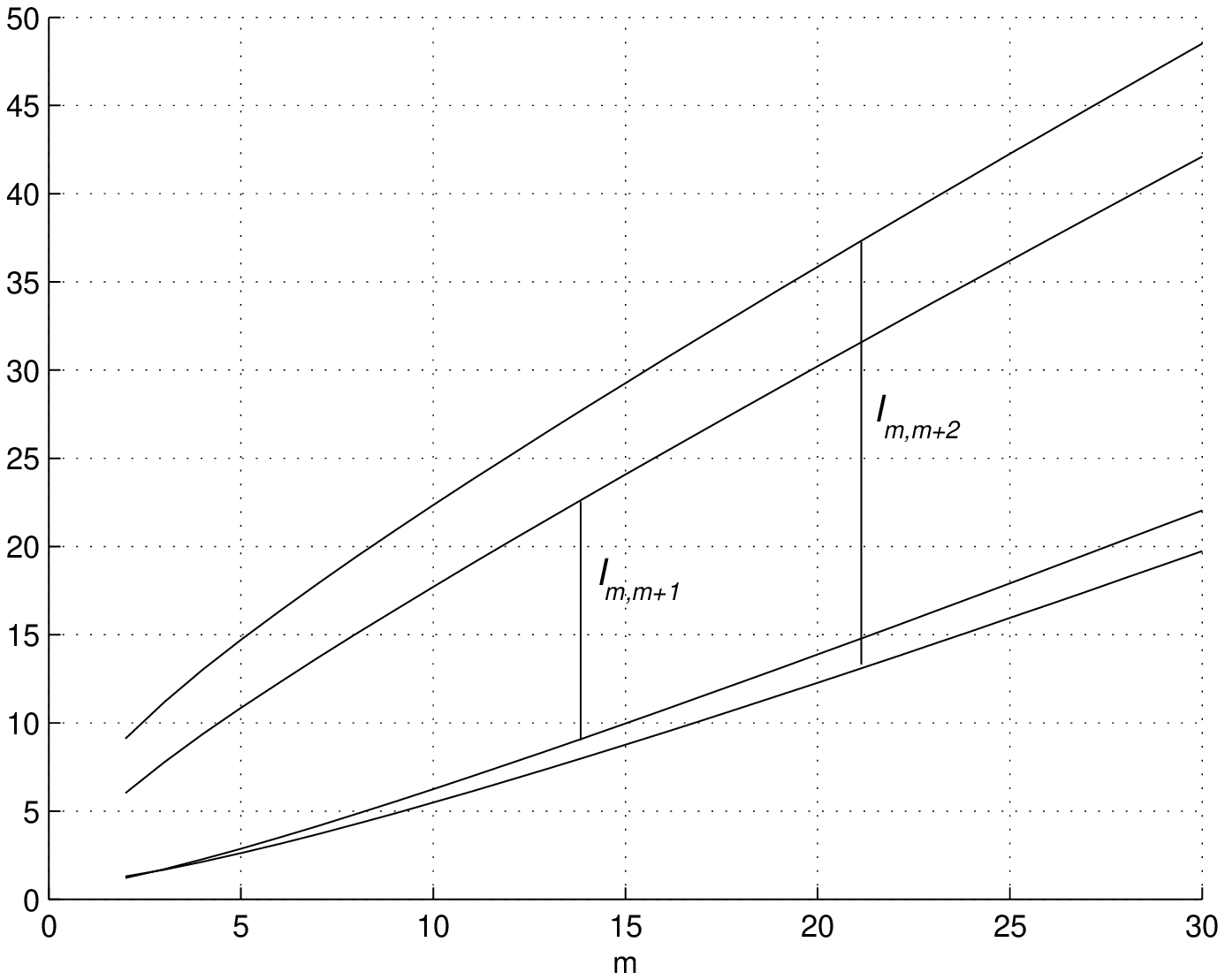}{\special{language
"Scientific Word";type "GRAPHIC";maintain-aspect-ratio TRUE;display
"USEDEF";valid_file "F";width 4.7314in;height 3.4541in;depth
0in;original-width 6.7196in;original-height 4.8948in;cropleft "0";croptop
"1";cropright "1";cropbottom "0";filename 'interval.eps';file-properties
"XNPEU";}}
\end{center}

We point out that the results are even a bit conservative with respect to
what happens in practice, and this is due to the approximation $%
h_{m+1,m}\approx 1/4$. Indeed larger intervals would be obtained taking $%
h_{m+1,m}<1/4$ as it occurs in practice.

In order prove the effectiveness of the above considerations let us consider
again the operator (\ref{lu}) with the usual discretization. We consider the
case $c=2$, $k=1$ for $h=0.1$. To define $\tau $ we consider again the
discretization with $M=50$ interior points observing the generalized
residual. This leads us to define $\tau =(m+k)/\cos \theta $ with $m=14$. In
Figure 5 we consider the behavior of the method for $\tau $, $\tau /2$ and $%
2\tau $.

\begin{center}
\FRAME{itbpFU}{3.4272in}{5.5192in}{0in}{\Qcb{Figure 5 - Error and error
bound (\protect\ref{fe1}) for $k=1$, $h=0.1$ and $L$ arising from (\protect
\ref{lu}) with $c=2$. Method applied with $\protect\tau =15/\cos \protect%
\theta $, $\protect\tau /2$ and $2\protect\tau $.}}{}{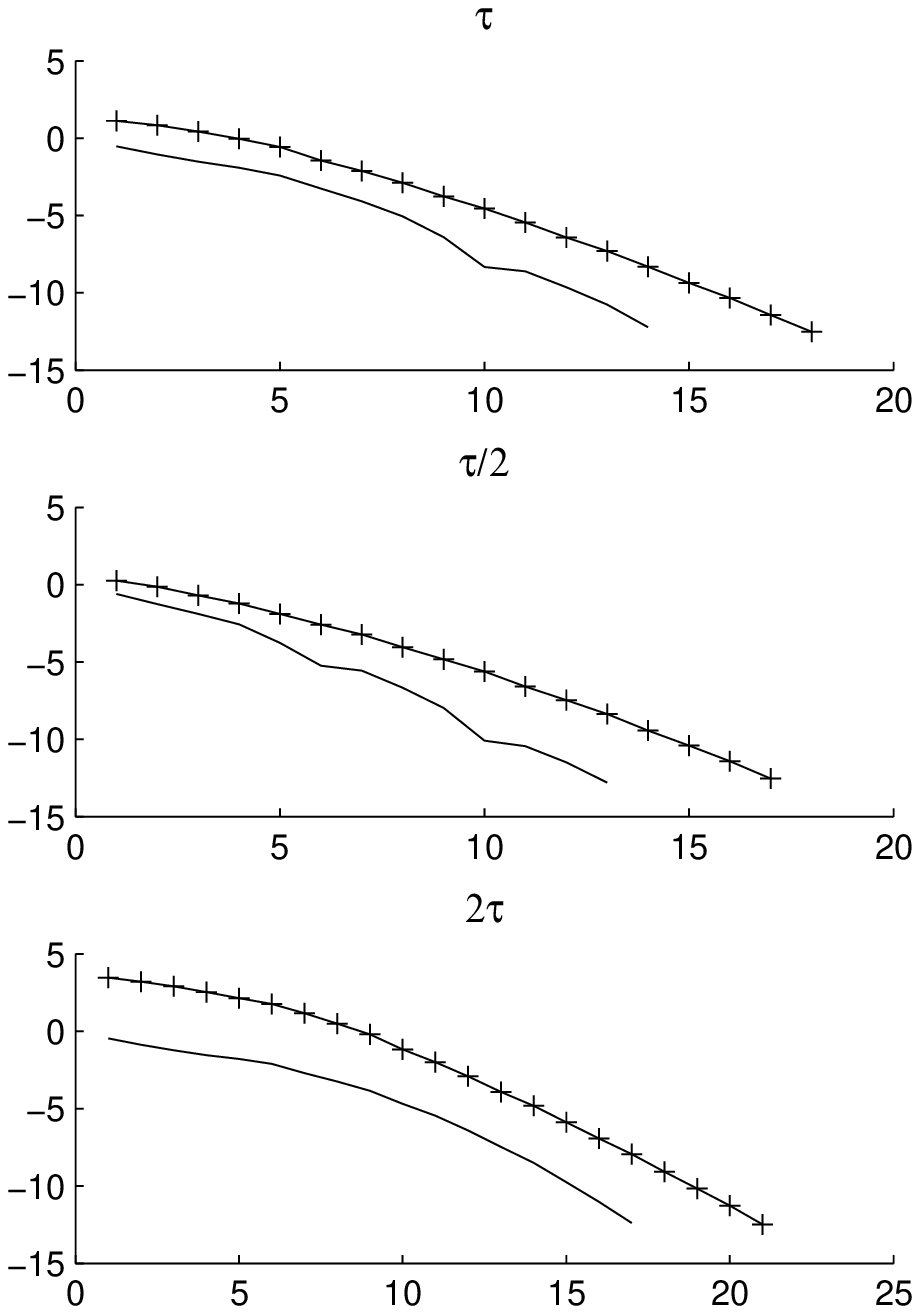}{\special%
{language "Scientific Word";type "GRAPHIC";maintain-aspect-ratio
TRUE;display "USEDEF";valid_file "F";width 3.4272in;height 5.5192in;depth
0in;original-width 4.2505in;original-height 6.8649in;cropleft "0";croptop
"1";cropright "1";cropbottom "0";filename 't01.eps';file-properties "XNPEU";}%
}
\end{center}

The robustness of the method with respect to the choice of $\tau $ is maybe
the most important aspect concerning its use inside an exponential
integrator. We want to give here some practical suggestions assuming to use
a sparse factorization technique to solve the linear systems with $I-\delta
L $, that, computationally, has to be considered the heaviest part of the
method.

\begin{enumerate}
\item Working in much smaller dimension compute $\theta $ and use the
generalized residual to estimate the initial $\tau _{opt}$.

\item For nonlinear problems, interpreting $L$ as the Jacobian of the system
(\cite{CO}, \cite{Tok}), it is necessary to introduce some strategies in
order to reduce as much as possible the number of updates of $L$ during the
integration, since each update would also require to update the
factorization. As for exponential W-method (see \cite{Holuse}, \cite{Nov}),
we suggest, whenever it is possible, to work with a time-lagged Jacobian and
hence to introduce the necessary order conditions in order to preserve the
theoretical order.

\item Using a quasi-constant step-size strategy (without Jacobian update)
allows to keep the factorization of $I-\delta L$ constant for a certain
number of steps. Whenever it is necessary to update the stepsize $%
h_{old}\rightarrow h_{new}$ without changing the Jacobian, if we want to
keep the previous factorization of $I-\delta _{old}L$ we just need to
consider the ratio $\tau =h_{new}/\delta _{old}$. If (indicatively) it is
bigger than $2\tau _{opt}$ or smaller than $\tau _{opt}/2$ (cf. Figure 4 and
5), where $\tau _{opt}$ arises from a previous analysis of the generalized
residual, then we need to update the factorization (cf. again \cite{Nov}),
otherwise we can keep the previous one. In this phase, however, one can even
considers other strategies to define suitably the window of admissible
values of $\tau $ around $\tau _{opt}$, taking into account of the local
accuracy required by the integrator, the norm of $v$, etc.
\end{enumerate}

\section{The superlinear decay of $\prod\nolimits_{i=1}^{m}h_{i+1,i}$}

Looking carefully at Figure 5 we notice that while the analysis in smaller
dimension suggested to take $\tau =15/\cos \theta $ for reaching the desired
tolerance in exactly $14$ iterations the method is unexpectedly a bit faster
taking $\tau _{1}=\tau /2$ (second picture). The analysis was correct
because in larger dimension the method actually achieves the tolerance in $%
14 $ iterations (first picture). In order to understand the reason of this
behavior, we need to remember that the definition of $\tau _{opt}=(m+k)/\cos
\theta $ given at the end of Section 4 was based on the assumption that $%
\prod\nolimits_{i=1}^{m}h_{i+1,i}$ is independent of $\delta ,$ but this is
not true. In what follows we try to provide a more accurate analysis
studying the decay of $\prod\nolimits_{i=1}^{m}h_{i+1,i}$.

We denote by $\sigma _{j}$, $j\geq 1$, the singular values of $Z$. Moreover
we denote by $\lambda _{j}$, $j\geq 1$ the eigenvalues of $Z$ and assume
that $\left\vert \lambda _{j}\right\vert \geq \left\vert \lambda
_{j+1}\right\vert $ for $j\geq 1$. We have the following result (cf. \cite%
{Ne} Theorem 5.8.10).

\begin{theorem}
\label{nev}Assume that $1\notin \sigma (Z)$ and%
\begin{equation}
\sum_{j\geq 1}\sigma _{j}^{p}<\infty \text{ for a certain }0<p\leq 1\text{. }
\label{psum}
\end{equation}%
Let $p_{m}(z)=\prod\nolimits_{i=1}^{m}(z-\lambda _{i})$. Then%
\begin{equation}
\left\Vert p_{m}(Z)\right\Vert \leq \left( \frac{\eta ep}{m}\right) ^{m/p},
\label{bn}
\end{equation}%
where%
\begin{equation*}
\eta \leq \frac{1+p}{p}\sum_{j\geq 1}\sigma _{j}^{p}.
\end{equation*}
\end{theorem}

As already shown in Section 4%
\begin{equation*}
\prod\nolimits_{i=1}^{m}h_{i+1,i}\leq \left\Vert p_{m}(Z)v\right\Vert
\end{equation*}%
for each monic polynomial $p_{m}$ of exact degree $m$ (see \cite{Trefe} p.
269), so that Theorem \ref{nev} reveals that the rate of decay of $%
\prod\nolimits_{i=1}^{m}h_{i+1,i}$ is superlinear and depends on the $p$%
-summability of the singular values of $Z$. We remark moreover that an
almost equal bound has been obtained in \cite{Han} studying the convergence
of the smallest Ritz value of the Lanczos process for self-adjoint compact
operators.

In practice, the use of (\ref{bn}) requires the knowledge of $p$ and a bound
for $\eta $, that is, information about the singular values of the operator $%
Z$. As a model problem we consider again the operator $L$ defined by (\ref%
{lu}) with $c=0$, whose eigenvalues are $\left( j\pi \right) ^{2}$, $j\geq 1$%
, so that the eigenvalues of $Z$ are given by $\lambda _{j}=1/(1+\delta
\left( j\pi \right) ^{2})$. In this case (\ref{psum}) holds for $1/2<p\leq 1$
so that $Z$ can be referred to as a \emph{trace class} operator (see again 
\cite{Ne}). Hence, taking for instance $p=1$ we have%
\begin{eqnarray}
\sum_{i\geq 1}\sigma _{i}^{p} &\leq &\frac{1}{\sqrt{\delta }}\left( \frac{1}{%
2}-\frac{\arctan \left( \sqrt{\delta }\pi \right) }{\pi }\right) ,  \notag \\
&=&\frac{1}{\sqrt{\delta }\pi }\arctan \left( \frac{1}{\sqrt{\delta }\pi }%
\right) ,  \notag \\
&\leq &\frac{1}{2\sqrt{\delta }},  \label{tb}
\end{eqnarray}%
and so%
\begin{equation}
\prod\nolimits_{i=1}^{m}h_{i+1,i}\leq \left\Vert p_{m}(Z)v\right\Vert \leq
\left( \frac{e}{\sqrt{\delta }m}\right) ^{m}.  \label{pmh}
\end{equation}%
The bound (\ref{pmh}) reveals that the rate of decay depends on the choice
of $\delta $ and then on $h$. For large values of $\delta $, say $\delta
\geq 1$, the bound (\ref{tb}) can be heavily improved exploiting the
properties of the $\arctan $ function and the convergence is extremely fast.
The following proposition states a general superlinear bound that can be
used when $L$ is an elliptic differential operator of the second order, so
with singular values growing like $j^{2}$. The proof is straightforward
since we just require to bound $\sum_{j\geq 1}\sigma _{j}^{p}$, and apply (%
\ref{bn}) with $p=1$.

\begin{proposition}
Let $L$ be an elliptic differential operator of the second order. Then there
exists a constant $C$ such that%
\begin{equation}
\prod\nolimits_{i=1}^{m}h_{i+1,i}\leq \left( \frac{C}{\sqrt{\delta }m}%
\right) ^{m}.  \label{hh}
\end{equation}
\end{proposition}

This proposition can easily be generalized to operator of order $s\geq 1$,
exploiting \cite{Ne} Corollary 5.8.12 in which the author extends Theorem %
\ref{nev} for $p>1$. Anyway, this is beyond the purpose of this section.

From a practical point of view, formula (\ref{hh}) is almost useless since
too much information on $L$ would be required. On the other side, it is
fundamental to understand the dependence on $\delta $. Setting as usual $%
\tau =h/\delta $ and putting the corresponding bound (\ref{hh}) in Theorem %
\ref{pro1} (formula (\ref{fe2})), we easily find that the theoretical
optimal value for $\tau $ is obtained seeking for the minimum of%
\begin{equation*}
\frac{e^{\tau \cos \theta -m-k-1}}{\tau ^{m+k}}\left( \frac{C\sqrt{\tau }}{%
\sqrt{h}m}\right) ^{m}
\end{equation*}%
with respect to $\tau $, that is,%
\begin{equation*}
\tau _{opt}=\frac{m+2k}{2\cos \theta }.
\end{equation*}%
This new value, less than $\left( m+k\right) /\cos \theta $, explains our
considerations about Figure 5 given at the beginning of this section.

We need to point out that since the choice of $\tau _{opt}$ is independent
of $C$ and $h$, formula (\ref{hh}) is quite coarse for small values of $h$
and not able to catch the fast decay of $\prod\nolimits_{i=1}^{m}h_{i+1,i}$.
In any case, if an estimate of $C$ is available an a-priori bound for the RD
Arnoldi method can be obtained taking%
\begin{equation*}
\prod\nolimits_{i=1}^{m}h_{i+1,i}\leq \min \left\{ \left( \frac{C\sqrt{\tau }%
}{\sqrt{h}m}\right) ^{m},2\left( \frac{1}{2(2-\nu )}\right) ^{m}\right\} ,
\end{equation*}%
(cf. (\ref{bh})). Consequently we argue that%
\begin{equation*}
\frac{m+2k}{2\cos \theta }\leq \tau _{opt}\leq \frac{m+k}{\cos \theta },
\end{equation*}%
with $\tau _{opt}$ close to $\left( m+2k\right) /\left( 2\cos \theta \right) 
$ for $h$ large and to $\left( m+k\right) /\cos \theta $ for $h$ small.

\section{Conclusions}

In this paper we have tried to provide all the necessary information to
employ the RD Arnoldi method as a tool for solving parabolic problems with
exponential integrators. The little number of codes available in literature,
and consequently, the little number of comparisons with classical solvers is
a source of skepticism about the practical usefulness of this kind of
integrators. Indeed, with respect to the most powerful classical methods for
stiff problems, the computation of a large number of matrix functions,
generally performed with a polynomial method, is still representing a
drawback because of the computational cost. The use of polynomial methods
for these computations may even be considered inadequate whenever we assume
to work with an arbitrarily sharp discretization of the operator, since this
would result in a problem of polynomial approximation in arbitrarily large
domains. For these reasons, the use of rational approximations as the one
here presented, should be considered a valid alternative because of the fast
rate of convergence and the mesh independence property, provided that we are
able to exploit suitably the robustness of the method with respect to the
choice of the poles, as explained in Section 8 for our case

\begin{acknowledgement}
The author is grateful to Igor Moret and Marco Vianello for many helpful
discussions and suggestions.
\end{acknowledgement}


\begin{thebibliography}{99}
\bibitem{Aste} M. Abramovitz, A. Stegun, \emph{Handbook of Mathematical
Functions}, Dover Publications, Inc., New York, 1965.

\bibitem{Beck} B. Beckermann, \emph{Image num\'{e}rique, GMRES et polyn\^{o}%
mes de Faber}, C. R. Math. Acad. Sci. Paris 340 (2005), pp. 855--860.

\bibitem{BR} B. Beckermann, L. Reichel, Error estimation and evaluation of
matrix functions via the Faber transform, SIAM J. Numer. Anal. 47 (2009),
pp. 3849-3883.

\bibitem{CO} M. Caliari, A. Ostermann, \emph{Implementation of exponential
Rosenbrock-type integrators}, Appl. Numer. Math. 59 (2009), pp. 568--581.

\bibitem{Cro} M. Crouzeix, \emph{Numerical range and numerical calculus in
Hilbert space}\textit{, }J. Functional Analysis, 244 (2007), pp. 668--690.

\bibitem{deb} C. de Boor, \emph{Divided differences}, Surveys in
approximation theory 1 (2005), pp. 46--69.

\bibitem{MRD} F. Diele, I. Moret, S. Ragni, \emph{Error estimates for
polynomial Krylov approximations to matrix functions}, SIAM J. Matrix
Analysis and Appl. 30, (2008), pp. 1546--1565.

\bibitem{Druki} V. Druskin, L. Knizhnerman, \emph{Extended Krylov subspaces:
approximation of the matrix square root and related functions}, SIAM J.
Matrix Anal. Appl. 19 (1998), pp. 755--771.

\bibitem{Dus} N. Dunford, J. T. Schwartz, \emph{Linear Operators, Part I},
John Wiley\& Sons, New York, 1963.

\bibitem{Ei1} M. Eiermann, \emph{Fields of values and iterative methods},
Linear Algebra Appl. 180 (1993) pp. 167--197.

\bibitem{Eier} M. Eiermann, O.G. Ernst, \emph{A restarted Krylov subspace
method for the evaluation of matrix functions}, SIAM J. Numer. Anal. 44
(2006), pp. 2481--2504

\bibitem{GalSa} E. Gallopoulos, Y. Saad, \emph{Efficient solution of
parabolic equations by Krylov approximation methods}, SIAM Sci. Stat.
Comput. 13 (1992), pp. 1236--1264.

\bibitem{Han} M. Hanke, \emph{Superlinear convergence rates for the Lanczos
method applied to elliptic operators}, Numer. Math. 77 (1997), pp. 487--499.

\bibitem{Hi} N. J. Higham, \emph{Matrix Computation Toolbox}. Version 1.2,
2002. www.mathworks.com.

\bibitem{Hig} N.J. Higham, \emph{Functions of matrices. Theory and
computation}. Society for Industrial and Applied Mathematics (SIAM),
Philadelphia, PA, 2008.

\bibitem{LubHoch} M. Hochbruck, C. Lubich, \emph{On Krylov subspace
approximations to the matrix exponential operator}, SIAM J. Numer. Anal. 34
(1997), pp. 1911--1925.

\bibitem{Holuse} M. Hochbruck, C. Lubich, H. Selhofer, \emph{Exponential
integrators for large systems of differential equations}, SIAM J. Sci.
Comput. 19 (1998), pp. 1552--1574.

\bibitem{HO} M. Hochbruck, A. Ostermann, \emph{Exponential integrators},
Acta Numer. 19 (2010), pp. 209--286.

\bibitem{Ka} T. Kato, \emph{Perturbation Theory for Linear Operators},
Springer, Berlin, 1976.

\bibitem{Knizh} L. Knizhnerman, \emph{Calculation of functions of
unsymmetric matrices using Arnoldi's method}, U.S.S.R. Comput. Maths. Math.
Phys. 31 (1991), pp. 1--9.

\bibitem{knisi} L. Knizhnerman, V. Simoncini, \emph{A new investigation of
the extended Krylov subspace method for matrix function evaluations}, Numer.
Linear Algebra Appl. 17 (2010), pp. 615--638.

\bibitem{Kopo} T. Kovari, C. Pomerenke, \emph{On Faber polynomials and Faber
expansions}, Math. Z. 99 (1967), pp. 193--206.

\bibitem{Mag} W. Magnus, F. Oberhettinger, R.P. Soni, \emph{Formulas and
theorems for the special functions of mathematical physics}. Third enlarged
edition. Springer-Verlag, New York, 1966

\bibitem{Minchev} B. V. Minchev, W. M. Wright, \emph{A review of exponential
integrators for first order semi-linear problems}, Preprint Numerics 2/2005,
Norwegian University of Science and Technology, Trondheim, Norway.

\bibitem{MN1} I. Moret, P. Novati, \emph{An interpolatory approximation of
the matrix exponential based on Faber polynomials}, Journal C.A.M. 131
(2001), pp. 361--380.

\bibitem{Morno} I. Moret, P. Novati, \emph{RD-rational approximations of the
matrix exponential}, BIT 44 (2004), pp. 595--615.

\bibitem{Mofi} I. Moret, \emph{On RD-rational Krylov approximations to the
core-functions of exponential integrators}, Numerical Linear Algebra with
Applications 14 (2007), pp. 445--457.

\bibitem{Ne} O. Nevanlinna, \emph{Convergence of Iterations for Linear
Equations}, Birkh\"{a}user, Basel, 1993.

\bibitem{Nors} S.P. Norsett, \emph{Restricted Pad\'{e} approximations to the
exponential function}, SIAM J. Numer. Anal. 15 (1978), pp. 1008--1029.

\bibitem{Nov} P. Novati, \emph{On the construction of Restricted-Denominator
Exponential W-methods}, Journal C.A.M. 221 (2008), pp. 86--101.

\bibitem{posi} M. Popolizio, V. Simoncini, \emph{Acceleration techniques for
approximating the matrix exponential operator}, SIAM J. Matrix Analysis and
Appl. 30 (2008), pp. 657--683.

\bibitem{Saad2} Y. Saad, \emph{Analysis of some Krylov subspace
approximations to the matrix exponential operator}, SIAM J. Numer. Anal. 29
(1992), pp. 209--228.

\bibitem{Smile} V. I. Smirnov, N. A. Lebedev, \emph{Functions of a Complex
Variable-Constructive Theory}, Iliffe Books, London, 1968.

\bibitem{Spik} M. N. Spijker, \emph{Numerical ranges and stability estimates}%
, Appl. Numer. Math. 13 (1993), pp. 241--249.

\bibitem{Tok} M. Tokman, \emph{Efficient integration of large stiff systems
of ODEs with exponential propagation iterative (EPI) methods}. J. Comput.
Phys. 213 (2006), pp. 748--776.

\bibitem{Trefe} L. N. Trefethen, D. Bau , \emph{Numerical Linear Algebra},
SIAM, Philadelphia, 1997.

\bibitem{Vanh} J. v. d. Eshof, M. Hochbruck, \emph{Preconditioning Lanczos
approximations to the matrix exponential}\textit{,} SIAM J. Sci. Comp. 27
(2005), pp. 1438--1457.

\bibitem{Wal} J.L. Walsh, \emph{Interpolation and Approximation by Rational
Functions in the Complex Domain}, AMS, Providence, 1965.

\bibitem{Zac} P. F. Zachlin, \emph{On the field of values of the inverse of
a matrix}, PhD Thesis, Case Western Reserve University, 2008.
\end{thebibliography}
\end{document}